\documentclass[a4paper,11pt]{article}
\usepackage{amsmath, amssymb, amsthm}
\usepackage{pstricks, pst-node}
\SpecialCoor
\usepackage[english]{babel}

\newtheorem{theorem}{Theorem}
\newtheorem{proposition}[theorem]{Proposition}
\newtheorem{definition}[theorem]{Definition}
\newtheorem{lemma}[theorem]{Lemma}
\newtheorem{corollary}[theorem]{Corollary}

\newcommand{\st}{\ \ |\ \ }
\newcommand{\A}{\mathcal{A}}
\newcommand{\B}{\mathcal{B}}
\newcommand{\C}{\mathcal{C}}
\newcommand{\D}{\mathbf{D}}
\newcommand{\F}{\mathcal{F}}
\newcommand{\N}{\mathbb{N}}
\newcommand{\Nbh}{\mathcal{N}}
\newcommand{\Pw}{\mathcal{P}}
\newcommand{\Q}{\mathcal{Q}}
\newcommand{\R}{\mathbb{R}}

\title{Sequences and nets in topology}
\author{Stijn Vermeeren (University of Leeds)}
\date{\today}

\begin{document}

\maketitle

In a metric space, such as the real numbers with their standard metric, a set $A$ is open if and only if no sequence with terms outside of $A$ has a limit inside $A$. Moreover, a metric space is compact if and only if every sequence has a converging subsequence. However, in a general topological space these equivalences may fail. Unfortunately this fact is sometimes overlooked in introductory courses on general topology, leaving many students with misconceptions, e.g. that compactness would always be equal to sequence compactness. The aim of this article is to show how sequences might fail to characterize topological properties such as openness, continuity and compactness correctly. Moreover, I will define nets and show how they succeed where sequences fail.

This article grew out of a discussion I had at the University of Leeds with fellow PhD students Phil Ellison and Naz Miheisi. It also incorporates some work I did while enrolled in a topology module taught by Paul Igodt at the Katholieke Universiteit Leuven in 2010.

\section{Prerequisites and terminology}

I will assume that you are familiar with the basics of topological and metric spaces. Introductory reading can be found in many books, such as \cite{kelley} and \cite{munkres}.

I will frequently refer to a topological space $(X, \tau)$ by just the unlying set $X$, when it is irrelevant or clear from the context which topology on $X$ is considered. Remember that any metric space $(X,d)$ has a topology whose basic opens are the open balls \[ B(x,\delta) = \{ y \st d(x,y) < \delta \} \] for all $x \in X$ and $\delta > 0$.

A \textbf{neighbourhood} of a point $x$ in a topological space is an open set $U$ with $x \in U$. Note that some people call $U$ a neighbourhood of $x$ if $U$ just contains an open set containing $x$ \cite[p.~97]{munkres}, but in this article neighbourhoods are always open themselves.

A \textbf{sequence} $(x_n)$ \textbf{converges} to a point $y$ if every neighbourhood of $y$ contains $x_n$ for $n$ large enough. We write $x_n \to y$ and say that $y$ is a \textbf{limit} of the sequence $(x_n)$. If $(x_n)$ converges to $y$, then so does every subsequence of $(x_n)$. If $f: X \to Y$ is a continuous function and $x_n \to y$ in $X$, then also $f(x_n) \to f(y)$ in $Y$. (We say that continuous functions \textbf{preserve} convergence of sequences.) Convergence in a product space is pointwise, i.e. a sequence $(x_n)$ in $\prod_{i \in I} X_i$ converges to $y$ if and only if $x_n(i) \to y(i)$ in $X_i$ for all $i \in I$.

A topological space is \textbf{Hausdorff} if for every two distinct points $x$ and $y$, we can find a neighbourhood of $x$ and a neighbourhood of $y$ that are disjoint. Sequences in general can have more than one limit, but in a Hausdorff space limits (if they exist at all) are unique. Indeed, a sequence cannot be eventually in two disjoint neighbourhoods at once.

A set $X$ is \textbf{countable} when there is a surjection from $\N = \{0, 1, 2, 3, \ldots \}$ onto $X$. So $X$ is countable if and only if $X$ is finite of in bijection with the natural numbers. A countable union of countable sets is still countable. Cantor's famous diagonal argument proves that the unit interval $[0,1]$ and $\R$ are  \textbf{uncountable}.\cite{cantor}

A few of my examples will make use of  \textbf{ordinal numbers}. If you are unfamiliar with ordinal numbers, you can either find background reading in \cite{hrbacek} or you can skip over these examples. I write the first infinite ordinal (i.e. the order type of the natural numbers) as $\omega_0$ and the first uncountable ordinal as $\omega_1$. Because a countable union of countable ordinals is still countable, no countable sequence of countable ordinals can have $\omega_1$ as limit. In other words: the cofinality of $\omega_1$ is $\omega_1$.

\section{Open versus sequentially open}

In a topological space $X$, a set $A$ is \textbf{open} if and only if every $a \in A$ has a neighbourhood contained in $A$. $A$ is \textbf{sequentially open} if and only if no sequence in $X \setminus A$ has a limit in $A$, i.e. sequences cannot \emph{converge out of} a sequentially closed set.


If $X$ is a metric space, then the two notions of open and sequentially open are equivalent. Indeed if $A$ is open, $(x_n)$ is a sequence in $X \setminus A$ and $y \in A$, then there is a neighbourhood $U$ of $y$ contained in $A$. Hence $U$ cannot contain any term of $(x_n)$, so $y$ is not a limit of the sequence and $A$ is sequentially open. Conversely, if $A$ is not open, then there is an $y \in A$ such that every neighbourhood of $y$ intersects $X \setminus A$. In particular we can pick an element \[ x_n \in (X \setminus A) \cap B(y, \frac{1}{n+1} ) \] for all $n \in \N$. The sequence $(x_n)$ in $X \setminus A$ then converges to $y \in A$, so $A$ is not sequentially open.

The implication from open to sequentially open is true in any topological space.

\begin{proposition}
\label{openseqopen}
In any topological space $X$, if $A$ is open, then $A$ is sequentially open.
\end{proposition}

We can just copy the proof for metric spaces, it remains valid in any topological space.

\begin{proof}
Suppose that $A$ is open, let $(x_n)$ be a sequence in $X \setminus A$ and take any $y \in A$. There is a neighbourhood of $y$ contained in $A$, so this neighbourhood doesn't contain any terms of $(x_n)$. Hence the sequence doesn't converge to $y$, as required.
\end{proof}

It is tempting to think that the converse might also hold in any topological space. When this is indeed the case, we call the space sequential.

\begin{definition}
A topological space is \textbf{sequential} when any set $A$ is open if and only if $A$ is sequentially open.
\end{definition}

However, importantly, not every space is sequential.

\begin{proposition}
\label{notseq}
There is a topological space that is not sequential.
\end{proposition}

\begin{proof}
Any of the three examples below constitutes a proof.

\begin{description}
\item{Example 1:} Let $X$ be an uncountable set, such as the set of real numbers. Consider $(X, \tau_{cc})$, the countable complement topology on $X$. Thus $A \subseteq X$ is closed if and only if $A=X$ of $A$ is countable. Suppose that a sequence $(x_n)$ has a limit $y$. Then the neighbourhood \[ \left(\R \setminus \{x_n \st n \in \N \} \right) \cup \{y\}\] of $y$ must contain $x_n$ for $n$ large enough. This is only possible if $x_n = y$ for $n$ large enough. Consequently a sequence in any set $A$ can only converge to an element of $A$, so every subset of $X$ is sequentially open. But as $X$ is uncountable, not every subset is open. So $(X, \tau_{cc})$ is not sequential.
\item{Example 2:} Consider the order topology on the ordinal $\omega_1 +1 = [0, \omega_1]$. Because $\omega_1$ has cofinality $\omega_1$, every sequence of countable ordinals has a countable supremum. Hence no sequence of countable ordinals converges to $\omega_1$, so $\{ \omega_1 \}$ is sequentially open. However, $\{ \omega_1 \}$ is not open as $\omega_1$ is a limit ordinal. So the order topology on $[0, \omega_1]$ is not sequential.
\item{Example 3:} Let $X$ be an uncountable set and let $\{0,1\}$ have the discrete topology. Consider $\Pw(X) = \{0,1\}^X$ with the product topology. Let $\A \subseteq \Pw(X)$ be the collection of all uncountable subsets of $X$. $\A$ is not open; indeed every basic open contains finite sets. However, we claim that $\A$ is sequentially open. Let $(X_n)$ be a sequence of countable subsets of $X$ and suppose that $X_n \to Y$. Then for every $x \in X$ we must have $x \in Y$ if and only if $x \in X_n$ for $n$ large enough. In particular \[ Y \subseteq \bigcup_{n \in \N} X_n.\] But $ \displaystyle \cup_{n \in \N} X_n$ is a countable union of countable sets. Hence a sequence of countable sets can only converge to countable sets, so $\A$ is sequentially open. \qedhere
\end{description}
\end{proof}

Still, a large class of topological spaces is sequential.

\begin{definition}
A \textbf{countable basis at a point $x$} is a countable set \linebreak\mbox{$\{U_n \st n \in \N \}$} of neighbourhoods of $x$, such that for any neighbourhood $V$ of $x$ there is an $n \in \N$ such that $U_n \subseteq V$.

A topological space is \textbf{first countable} if every point has a countable basis.
\end{definition}

Every metric space is first countable, as 
\[\left\{ B\left(x, \frac{1}{n+1}\right) \st n \in \N  \right\} \]
is a countable basis at any point $x$. We can prove that every first countable space is sequential by generalizing the proof that every metric space is sequential.

\begin{proposition}
\label{firstcountableseq}
Every first countable space $X$ (and hence every metric space) is sequential.
\end{proposition}

\begin{proof}
Because of Proposition \ref{openseqopen}, we only need to prove that every sequentially open set $A$ is also open. So suppose that $A$ is not open. Then there is an $y \in A$ such that every neighbourhood of $y$ intersects $X \setminus A$. Let $\{U_n \st n \in \N\}$ be a countable basis at $y$. For every $n\in \N$, we can choose 
\[ x_n \in (X \setminus A) \cap \left( \bigcap_{i=0}^n U_i \right). \]
Then for every neighbourhood $V$ of $y$, there is an $n \in \N$ such that $U_n \subseteq V$, and hence $x_m \in V$ for every $m \geq n$. So $(x_n)$ is a sequence in $X \setminus A$ that converges to $y \in A$. Therefore $A$ is not sequentially open, as required.
\end{proof}

Sequential spaces are also exactly those spaces $X$ where sequences can correctly define continuity of functions from $X$ into another topological space.

\begin{lemma}
\label{lemmaseq}
Let $X$ be a topological space. Then $A \subseteq X$ is sequentially open if and only if every sequence with a limit in $A$ has all but finitely many terms in $A$.
\end{lemma}
\begin{proof}
We prove that $A \subseteq X$ is \emph{not} sequentially open if and only if there is a sequence with infinitely many terms in $X \setminus A$ and with a limit in $A$.

If $A$ is not sequentially open, then by definition there is a sequence with terms in $X \setminus A$ but with limit in $A$.

Conversely, suppose $(x_n)$ is a sequence with infinitely many terms in $X \setminus A$ that converges to $y \in A$. Then $(x_n)$ has a subsequence in $X \setminus A$ that must still converge to $y \in A$, so $A$ is not sequentially open.
\end{proof}

\begin{proposition}
The following are equivalent for any topological space $X$:
	\begin{enumerate}
	\item $X$ is sequential;
	\item for any topological space $Y$ and function $f: X \to Y$, $f$ is continuous if and only if $f$ preserves convergence (i.e. whenever $x_n \to y$ in $X$, also $f(x_n) \to f(y)$ in $Y$).
	\end{enumerate}
\end{proposition}

\begin{proof}
\begin{description}
\item[$1 \Rightarrow 2$: ] Suppose $X$ is sequential. Any continuous function preserves convergence of sequences, so we only need to prove that if $f: X \to Y$ preserves convergence, then $f$ is continuous. Suppose for contradiction that $f$ is not continuous. Then there is an open $U \subseteq Y$ such that $f^{-1}(U)$ is not open in $X$. As $X$ is sequential, $f^{-1}(U)$ is also not sequentially open, so there is a sequence $(x_n)$ in $X \setminus f^{-1}(U)$ that converges to an $y \in f^{-1}(U)$. However $\left(f(x_n)\right)$ is then a sequence in the closed set $Y \setminus U$, so it cannot have $f(y)$ as a limit. So $f$ does not preserve convergence, as required.
\item[$2 \Rightarrow 1$: ] Suppose that the topological space $(X, \tau)$ is not sequential. Let $(X, \tau_{seq})$ be the topological space where $A \subseteq X$ is open if and only if $A$ is sequentially open in $(X, \tau)$. This is indeed a topology: it is trivial that $\emptyset$ and $X$ are sequentially open, and that any union of sequentially open sets is also sequentially open. It remains to prove that the intersection of two sequentially open sets $A$ and $B$ is sequentially open. Suppose that $(x_n)$ is a sequence with limit $y \in A \cap B$. By Lemma \ref{lemmaseq}, $(x_n)$ must have all but finitely many terms in $A$ and all but finitely many terms in $B$. So $(x_n)$ has all but finitely many terms in $A \cap B$. By Lemmma \ref{lemmaseq} again,  $A \cap B$ is sequentially open.

As $(X, \tau)$ is not already sequential, the topology $\tau_{seq}$ is strictly finer than $\tau$. Hence the identity map 
\[id : (X, \tau) \to (X, \tau_{seq}) \]
is not continuous. We claim that $f$ nonetheless preserves convergence. Indeed, suppose $x_n \to y$ in $(X, \tau)$. Every open neighbourhood of $y$ in $(X, \tau_{seq})$ is sequentially open in $(X, \tau)$, so by Lemma \ref{lemmaseq} contains all but finitely many terms of $(x_n)$. Hence also $x_n \to y$ in $(X, \tau_{seq})$, as required.\qedhere
\end{description}
\end{proof}

\section{Sequential spaces as quotients of metric spaces}

First, we recall the definition of a quotient space. Let $X$ be a topological space and let $\sim$ be an equivalence relation on $X$. Consider the set of equivalence classes $X / \sim$ and the projection map $\pi: X \to X/\sim$. We topologize $X / \sim$ by defining $A \subseteq X / \sim$ to be open if and only if $\pi^{-1}(A)$ is open in $X$.

Note that given a surjective function $f: X \to Y$ such that $A \subseteq Y$ is open if and only if $f^{-1}(A)$ is open in $X$, we can consider $Y$ to be a quotient of $X$. Indeed, define an equivalence relation $\sim$ on $X$ such that $x \sim y$ if and only if $f(x) = f(y)$, i.e. the equivalence classes are the fibers of $f$. Then $X / \sim$ is isomorphic to $Y$ by mapping the equivalence class of $x$ to $f(x)$.

We are now ready to prove that the sequential spaces are exactly the quotients of metric spaces. This is a corollary of the following two propositions.

\begin{proposition}
\label{quotientseq}
Any quotient $X / \sim$ of a sequential space $X$ is sequential.
\end{proposition}

\begin{proof}
Suppose that $A \subseteq X / \sim$ is not open. We need to prove that $A$ is not sequentially open either. By definition of quotient space, $\pi^{-1}(A)$ is not open in $X$. As $X$ is sequential, there is a sequence $(x_n)$ in $X \setminus \pi^{-1}(A)$ that converges to some $y \in \pi^{-1}(A)$. But $\pi$ is continuous, so it preserves convergence of sequences. Hence $(\pi(x_n))$ is a sequence in $(X/\sim) \setminus A$ with limit $\pi(y) \in A$. Thus $A$ is not sequentially open, as required.
\end{proof}

\begin{proposition}[Franklin \cite{franklin}]
\label{franklin}
Every sequential space $X$ is a quotient of some metric space.
\end{proposition}

\begin{proof}
Let $\C$ be the set of all sequences $(x_n)$ in $X$ that converge to their first term, i.e. $x_n \to x_0$.

Consider the subspace $Y = \{0\} \cup \{\frac{1}{n+1} \st n \in \N \}$ of $\R$ with the standard metric. Thus, $A \subseteq Y$ is open if and only if $0 \not\in A$ or $A$ contains all but finitely many elements of $Y$. Note that $Y$ is metric as a subspace of a metric space.

Now consider the disjoint sum (i.e. the coproduct in category theory jargon) 
\[ Z = \bigoplus_{(x_n) \in \C} \{(x_n)\} \times Y. \]
The underlying set of $Z$ is 
\[ \bigcup_{(x_n) \in \C} \{(x_n)\} \times Y \]
 and $A \subseteq Z$ is open if and only if for every $(x_n) \in \C$ the set
\[ \{ y \in Y \st ((x_n), y) \in A\} \]
is open in $Y$. Note that $Z$ is metrizable a disjoint sum of metric spaces.

Next consider the map 
\begin{align*}
f: \  Z &\to X \\
((x_n),0) &\mapsto x_0 \\
\left((x_n),\frac{1}{i+1}\right) &\mapsto x_i \quad \quad \quad \mbox{for all $i \in \N$}
\end{align*}
I claim that this map exhibits $X$ as a quotient of $Z$. Indeed $f$ is clearly surjective: for all $x \in X$ the constant sequence at $x$ converges to $x$, so $x=f((x),0)$. Hence it remains to show that $A \subseteq X$ is open if and only if $f^{-1}(A)$ is open in $Z$.

Suppose that $A \subseteq X$ is open. As $X$ is sequential, every sequence $(x_n)$ in $X$ converging to some $a \in A$ must have all but finitely many terms in $A$ by Lemma \ref{lemmaseq}. So if $((x_n), 0) \in f^{-1}(A)$ (which means that $(x_n)$ converges to $x_0 \in A$), then $f^{-1}(A)$ will contain all but finitely many elements of $\{(x_n)\} \times Y$. So for each $(x_n) \in \C$, the set
\[ \{ y \in Y \st ((x_n), y) \in f^{-1}(A) \} \]
is open in $Y$. Hence $f^{-1}(A)$ is open in $Z$. 

Conversely, if $A$ is not open in $X$, then there is some sequence $(x_n)$ in $X \setminus A$ that converges to some $a \in A$. But then
\[ \{ y \in Y \st ((x_n), y) \in f^{-1}(A) \} = \{0\} \]
is not open in $Y$, so $f^{-1}(A)$ is not open in $Z$.
\end{proof}

\begin{corollary}
\label{seqquotientmetric}
A topological space is sequential if and only if it is a quotient of a metric space.
\end{corollary}

\begin{proof}
One direction is the above proposition. For the other direction, note that by Proposition \ref{firstcountableseq} every metric space is sequential, so by Proposition \ref{quotientseq} any quotient of a metric space is also sequential.
\end{proof}

We can now also easily prove that sequential is a strictly weaker notion than first countable.

\begin{proposition}
\label{seqnotfirstcountable}
There exists a sequential space which is not first countable.
\end{proposition}

\begin{proof}
Consider $\R$ with the standard topology, and the equivalence relation $\sim$ on $\R$ that identifies all the natural numbers, i.e. the equivalence classes are $\N$ and $\{x\}$ for every $x \in \R \setminus \N$.

The quotient space $\R / \sim$ is sequential as a quotient of a metric space.

I claim that $\R / \sim$ is not first countable, in particular that there is no countable basis at $\N$. Suppose that $\{U_n \st n \in \N \}$ is any countable collection of neighbourhoods of $\N$. For all $n \in \N$, $\pi^{-1}(U_n)$ is a neighbourhood of $n$ in $\R$ with the standard topology, so there is a $\delta_n > 0$ such that $B(n, \delta_n) \subseteq \pi^{-1}(U_n)$. Then consider 
\[ \pi\left(\bigcup_{n \in \N} B\left(n, \frac{\delta_n}{2} \right) \right). \]
This is a neighbourhood of $\N$ in $X / \sim$, but it doesn't contain $U_n$ for any $n \in \N$. So $\{U_n \st n \in \N \}$ is not a countable basis at $\N$, as required.
\end{proof}

\section{Nets save the day}

Looking back at the section on open and sequentially open sets, we see that convergence of sequences doesn't give us full information on the topology. For example the discrete topology and the countable complement topology on an uncountable set $X$ have the same converging sequences (namely $x_n \to y$ if and only if $x_n = y$ for $n$ large enough), but the discrete topology is strictly finer than the countable complement topology. The discrete topology is sequential, but the countable complement topology contains sequentially open sets which are not open.

Convergence of sequences works fine when the space is first countable, because a countable basis at a point allows us to \emph{approach} that point nicely with a sequence. However, if a point $x$ does not have a countable basis, then sequences might not succeed in getting \emph{close to} $x$, i.e. eventually in every neighbourhood of $x$. Sequences fall short in two respects: they are \emph{too short} and \emph{too thin}.

Remember Example 2 from the proof of Proposition \ref{notseq}; $[0, \omega_1]$ with the order topology. Even though $\{\omega_1\}$ is not open, a sequence of countable ordinals can only have a countable limit. Because a sequence only has countable many terms, it never can advance \emph{deep enough} in the ordinal numbers to get close to $\omega_1$. A possible solution is to allow sequences indexed by any linearly ordered set, instead of just the natural numbers. Indeed the $\omega_1$-sequence 
\begin{align*}
x: \omega_1 &\to [0, \omega_1] \\
	\alpha &\mapsto \alpha
\end{align*}
converges to $\omega_1$, even though every term is countable.

This overcomes the \emph{shortness} of sequences, but is still not enough to solve all difficulties. Indeed, reconsider Example 3 from the proof of Proposition \ref{notseq}; the product space $\Pw(X) = \{0,1\}^X$ where $X$ is uncountable. We take the subspace $\Q$ of $\Pw(X)$ which consists of those subsets of $X$ that are either finite or uncountable. Define $\A \subseteq \Pw(X)$ to be the collection of all uncountable subsets of $X$ like before. $\A$ is still not open in $\Q$ as every basic open contains finite sets. But no sequence of finite sets can converge to an uncountable set, not even if we allow sequences indexed by any linearly ordered set. Indeed, let $(X_i)_{i\in I}$ be any $I$-sequence of finite subsets of $X$, where $I$ is any linearly ordered set. For all $i \in I$, define 
\[ Z_i = \{x \in X \st \mbox{$x \in X_j$ for all  $j \geq i$} \}. \]
Every $Z_i$ is finite as $Z_i \subseteq X_i$. Moreover $Z_i \subseteq Z_j$ for $i \leq j$, so there can only be countably many distinct sets $Z_i$. (There can be at most one $Z_i$ with cardinality $1$, at most one with cardinality $2$, etc.) Hence $\cup_{i \in I} Z_i$ is countable as a countable union of countable sets. But if $X_i \to Y$ then we must have $Y = \cup_{i \in I} Z_i$, so $Y$ cannot be uncountable. Intuitively, the problem with sequences here is that they are linearly ordered, so they can only approach a point from \emph{one angle at a time}, whereas to capture the topology we need to consider all \emph{angles of approach} at the same time.

Nets are defined to overcome the shortcomings of sequences. Nets generalize sequences, but they can go both \emph{deeper} and \emph{wider} than sequences. Sequences associate a point to every natural number. Nets are more general, as they can associate a point to every element of a directed set.

\begin{definition}
A \textbf{directed set} is a set $D$ with a preorder relation (i.e. a reflexive and transitive binary relation) such that every two elements have an upper bound.
\end{definition}

Note that we \emph{don't} require that a pair of elements has a \emph{least} upper bound, we just require that some upper bound exists.

{
\setlength{\leftmargini}{5em} 
\begin{enumerate}
\item[Example 4]: Every linearly ordered set (such as the $\N$ with the usual order) is a directed set.
\item[Example 5]: Any collection of sets that is closed under binary intersections is a directed set when ordered by reverse inclusion, i.e. $X \le Y$ if and only if $Y \subseteq X$.

In particular, given any point $x$ of a topological space, the collection of all neighbourhoods of $x$ ordered by reverse inclusion is a directed set, which we write as $\Nbh(x)$.
\item[Example 6]: If $D$ and $E$ are directed sets, then so is their product $D \times E$ ordered by $(d_1,e_1) \le (d_2,e_2)$ if and only if $d_1 \le d_2$ in $D$ and $e_1 \le e_2$ in $E$.
\end{enumerate}
}

\begin{definition}
A \textbf{net} in a topological space $X$ is a function $f$ from a directed set $D$ to $X$. We usually write $f(d) = x_d$ for all $d \in D$, an refer to the net by $(x_d)_{d\in D}$.

A net $(x_d)_{d\in D}$ \textbf{converges} to a point $y \in X$ if for every neighbourhood $U$ of $y$, there is a $d \in D$ such that $x_e \in U$ for all $e \ge d$.
\end{definition}

Our motivation for defining nets was the hope that convergence of nets (in contrast to convergence of sequences) would completely determine the topology of the space. We now prove that this is indeed the case.

\begin{proposition}
\label{netsopen}
In any topological space $X$, a set $A \subseteq X$ is open if and only if no net in $X \setminus A$ has a limit in $A$.
\end{proposition}

\begin{proof}
Let $A$ be open, let $(x_d)_{d\in D}$ be a net in $X \setminus A$ and take any $y \in A$. As $A$ is open, there is a neighbourhood of $y$ that is contained in $A$. Hence this neighbourhood does not contain any terms of the net, so $y$ is not a limit of $(x_d)$.

Conversely suppose that $A$ is not open. Then there is an $y \in A$ such that every neighbourhood of $y$ intersects $X \setminus A$. So there is a net $(x_U)_{U \in \Nbh(y)}$ such that \[ x_U \in (X \setminus A) \cap U \] for all neighbourhoods $U$ of $y$. For every neighbourhood $U$ of $y$, $U$ is an element of $\Nbh(y)$ such that $x_V \in U$ for every $V \ge U$ (i.e. $V \subseteq U$). Hence $(x_U)$ is a net in $X \setminus A$ that converges to $y \in A$, as required.
\end{proof}

Like for sequences, convergence of nets is preserved by continuous functions. But again, for nets the converse if also true.

\begin{proposition}
\label{netscontinuous}
Let $f: X \to Y$ be a function between two topological spaces. Then $f$ is continuous if and only if for every net $(x_d)_{d\in D}$ that converges to $y$ in $X$, we have $f(x_d)_{d\in D} \to f(y)$ in $Y$
\end{proposition}

\begin{proof}
Suppose that $f$ is continuous and $x_d \to y$ in X. Take a neighbourhood $U$ of $f(y)$ in $Y$. Then $f^{-1}(U)$ is a neighbourhood of $y$ in $X$. By definition of convergence of nets, there is a $d \in D$ such that $x_e \in f^{-1}(U)$ for all $e \ge d$. So also $f(x_e) \in U$ for all $e \ge d$. This means that $f(x_d) \to f(y)$, as required.

Conversely, suppose that $f$ is not continuous, say $U \subseteq Y$ is open but $f^{-1}(U)$ is not open. By Proposition \ref{netsopen} there is a net $(x_d)_{d\in D}$ in $X \setminus f^{-1}(U)$ that converges to some $y \in f^{-1}(U)$. But then $(f(x_d))_{d\in D}$ is a net in the closed set $Y \subseteq U$ which (again by Proposition \ref{netsopen}) cannot converge to $f(y) \in U$. So $f$ doesn't preserve convergence of nets.
\end{proof}

Like a sequence, a net can have more than one limit, although in Hausdorff spaces every converging sequence has a unique limit. Sequences can also have unique limits in spaces that are not Hausdorff. Consider for example an uncountable set $X$ with the countable complement topology. In Example 1 we saw that limits are unique, but the space is obviously not Hausdorff. On the contrary, nets do succeed in exactly characterizing Hausdorff spaces.

\begin{proposition}
A space $X$ is Hausdorff is and only if no net has two distinct limits.
\end{proposition}

\begin{proof}
Suppose that $X$ is Hausdorff and consider a net $(x_d)_{d\in D}$. Suppose for contradiction that $x$ and $y$ are distinct limits of $(x_d)$. Take disjoint neighbourhoods $U$ of $x$ and $V$ of $y$. By definition of convergence, there is a $d_x$ such that $x_e \in U$ for all $e \ge d_x$ and a $d_y$ such that $x_e \in V$ for all $e \ge d_y$. In particular we have $x_e \in U \cap V$ for an upper bound $e$ of $d_x$ and $d_y$ in the directed set $D$, contradicting the disjointness of $U$ and $V$. Thus $(x_d)$ cannot have two distinct limits.

Conversely, suppose that $X$ is not Hausdorff, so there are two distinct points $x$ and $y$ such that any neighbourhood of $x$ intersects any neighbourhood of $y$. So there is a net $(x_{(U,V)})_{(U,V) \in \Nbh(x) \times \Nbh(y)}$ such that \[ x_{(U,V)} \in U \cap V \] for any neighbourhoods $U$ of $x$ and $V$ of $y$. Take any neighbourhood $U_0$ of $x$ and any $(U,V) \in \Nbh(x) \times \Nbh(y)$ with $(U,V) \ge (U_0, X)$. By definition we have $U \subseteq U_0$ and thus $x_{(U,V)} \in U \cap V \subseteq U_0$. This proves that $x_{(U,V)} \to x$ and we can similarly show that $x_{(U,V)} \to y$. So the net $(x_{(U,V)})_{(U,V) \in \Nbh(x) \times \Nbh(y)}$ has two distinct limits, as required.\end{proof}

\section{Compactness and sequential compactness}

Let's now look at compactness of topological spaces. Remember that a space $X$ is \textbf{compact} if and only if every open covering of $X$ (i.e. every collection of open sets whose union is $X$) has a finite subcovering. It is sufficient to consider coverings of basic opens.

Equivalently, $X$ is compact if and only if every collection of closed sets with the finite intersection property (i.e. all finite intersections are nonempty) has a nonempty intersection.

An important theorem by Tychonoff \cite{tychonoff} says that any product of compact spaces is itself compact.

If every sequence in a topological space has a convergent subsequence, then we call the space \textbf{sequentially compact}. A metric space is compact if and only if it is sequentially compact (\cite{munkres}, theorem 28.2). However, bearing in mind the the difference between open and sequentially open, we should be very suspicious of this equivalence holding in general. And indeed, neither direction of the equivalence holds in every topological space.

\begin{proposition}
\label{compactnotseqcompact}
There is a topological space that is compact but not sequentially compact.
\end{proposition}

\begin{proof}
Let $\{0,1\}$ have the discrete topology and consider $\{0,1\}^{[0,1)}$ with the product topology.

By Tychonoff's theorem, $\{0,1\}^{[0,1)}$ is compact as product of compact spaces.

Consider the sequence $(x_n)$ where $x_n(r)$ equals the $n$'th digit in the binary expansion of $r$ for all $r \in [0,1)$, where we never pick an expansion that ends in all $1$s. We claim that $(x_n)$ does not have a convergent subsequence. Indeed if the subsequence $x_{k_n}$ has a limit $y$, then for any $r \in [0,1)$, we must have $x_{k_n}(r) = y(r)$ for $n$ large enough. But there is a real number $r \in [0,1]$ whose unique binary expansion has a $0$ in the $k_n$th position if $n$ is even and a $1$ in the $k_n$th position is $n$ is odd, contradicting that $x_{k_n}$ has any limit. Hence $\{0,1\}^{[0,1)}$ is not sequentially compact.
\end{proof}

\begin{proposition}
\label{seqcompactnotcompact}
There is a topological space that is sequentially compact but not compact.
\end{proposition}

\begin{proof}
Consider the order topology on $\omega_1 = [0, \omega_1)$, the set of all countable ordinals.

The open covering $\A = \{[0, \alpha) \st \alpha < \omega_1\}$ does not have a finite subcovering. Indeed the supremum of the sets in any finite subset $\mathcal{B}$ of $\A$ is a finite union of countable ordinals, and hence itself a countable ordinal which has a successor in $[0, \alpha)$ that is not covered by $\mathcal{B}$.

Now let $(x_n)$ be any sequence in $[0, \omega_1)$. Let $A \subseteq \N$ be the set of all indices $n$ such that $x_n$ is minimal in $\{x_m \st m \geq n \}$. Note that $A$ is infinite as by definition it cannot have a largest element. Let $k_n$ be the $n$'th element of $A$ for all $n \in \N$. The sequence $(x_{k_n})$ is then a subsequence of $(x_n)$, which is nondecreasing and therefore converges to the supremum of its elements. This supremum is countable as a countable union of countable ordinals, so it is indeed an element of $[0, \omega_1)$. Thus every sequence in $[0, \omega_1)$ has a convergent subsequence, as required.
\end{proof}

We can however characterize compact spaces using nets. To do this, we need the notion of a subnet.

\begin{definition}
Let $(x_d)_{d \in D}$ be a net. A \textbf{subnet} of $(x_d)_{d \in D}$ is a net $(x_{f(e)})_{e \in E}$ where $E$ is a directed set and $f: E \to D$ is a function such that:
\begin{enumerate}
\item if $e_1 \le e_2$, then $f(e_1) \le f(e_2)$ (i.e. $f$ is order preserving),
\item for all $d \in D$, there is an $e \in E$ such that $f(e) \ge d$ (i.e. $f(E)$ is cofinal in $D$).
\end{enumerate}
\end{definition}

\begin{proposition}
A topological space $X$ is compact if and only every net has a convergent subnet.
\end{proposition}

\begin{proof}
Suppose $X$ is compact and let $(x_d)_{d \in D}$ be a net. The sets \[ X_d = \{x_e \st e \ge d\} \] have the finite intersection property, so by compactness their closures have a nonempty intersection. Thus we can take an $y \in \cup_{d \in D} \overline{X_d}$. Every neighbourhood $U$ of $y$ intersects $X_d$ for all $d \in D$. In other words, for all $d \in D$ there is an $e \ge d$ such that $x_e \in U$. (Such an $y$ is called a \textbf{cluster point} of the net $(x_d)$.)

Consider the set 
\[ E = \{ (d,U) \in D \times \Nbh(y) \st x_d \in U \}, \]
preordered by $(d_1, U_1) \le (d_2, U_2)$ if and only if $d_1 \le d_2$ in $D$ and $U_2 \subseteq U_1$. We claim that this is a directed set. Indeed let $(d_1, U_1)$ and $(d_2, U_2)$ be any pair of elements of $E$. There is an upper bound $e$ of $d_1$ and $d_2$ in $D$, and as $y$ is a cluster point of $(x_d)$, there is an $e' \ge e$ such that $x_{e'} \in U_1 \cap U_2$. Then $(e', U_1 \cap U_2)$ is an upper bound of $(d_1, U_1)$ and $(d_2, U_2)$ in $E$. Therefore $E$ is indeed a directed set.

Define the projection 
\begin{align*}
f:\quad E &\to D \\
   (d, U) &\mapsto d.
\end{align*}
This is clearly an order preserving function and it is a surjection as $d = d(d, X)$ for all $d \in D$. So $(x_{f(d,U)})_{(d,U)\in E}$ is a subnet of $(x_d)_{d \in D}$. Moreover, if $U$ is any neighbourhood of $y$, then there is by choice of $y$ a $d \in D$ such that $x_d \in U$. By definition of $E$ we have $x_{f(e,V)} \in U$ for all $(e, V) \ge (d, U)$. Therefore $x_{f(d,U)} \to y$, as required.

\underline{Conversely} suppose that $X$ is not compact. Then there is a collection $\A$ of closed sets with the finite intersection property, but with empty intersection. Let $\D$ be the set of finite subcollections of $A$, ordered by inclusion. This is clearly a directed set. We can choose a net $(x_\B)_{\B \in \D}$ where \[ x_\B \in \cap_{A \in \B} A \] for every finite $\B \subseteq \A$. Suppose for contradiction that $(x_{f(x)})_{e \in E}$ is a subnet of $(x_\B)_{\B \in \D}$ that converges to some $y \in X$. By assumption, there is an $A \in \A$ such that $y \not\in A$. As $A$ is closed, there is neighbourhood $U$ of $y$ such that $U \cap A = \emptyset$, and hence $x_\B \not\in U$ for all $\B \ge \{A\}$. As $f(E)$ is cofinal in $D$, there is an $e_1 \in E$ such that $\{A\} \le f(e_1)$. But there must also be an $e_2 \in E$ such that $x_{f(e)} \in U$ for all $e \ge e_2$. Let $e$ be an upper bound of $e_1$ and $e_2$. Then we must have $x_{f(e)} \in U$, but on the other hand $\{A\} \le f(e_1) \le f(e)$ so we must have $x_{f(e)} \not\in U$, a contradiction. So the net $(x_\B)_{\B \in \D}$ does not have a convergent subnet.
\end{proof}

It is now tempting to accept the following argument: \emph{In a compact space, every nets has a convergent subnet. So every sequence, considered as a net, has a convergent subnet, which is a convergent subsequence. So every compact space is sequentially compact.} But we know from Proposition \ref{compactnotseqcompact} that this is not true. The mistake is the fact that not every subnet of a sequence is a subsequence. In particular note that the function $f$ in the definition of a subnet need not be injective.

\section{Further reading}

\subsection{Nets and filters}

As an alternative to nets in $X$, one can consider filters on $X$. A \textbf{filter} on $X$ is a nonempty collection of subsets of $X$ that is closed under binary intersections and supersets, and does not contain $\emptyset$. A filter $\F$ \textbf{converges} to a point $x$ if $U \in \F$ for every neighbourhood $U$ of $x$. Given a net $(x_d)_{d \in D}$ in $X$ one can consider the filter
\[ \F = \{ A \subseteq X \st \exists d \in D: \mbox{$x_e \in A$ for all $e \geq d$} \} \]
of sets which \emph{eventually} contain all point of the net. This filter has the same limits as $(x_d)_{d \in D}$. Conversely given a filter $\F$ on $X$ one can consider the directed set $\F$ ordered by reversed inclusion. Then $\F$ converges to a point if and only if any net
\[ (x_A)_{A \in \F} \]
with
\[ x_A \in A \]
for all $A \in F$ converges to that point as well. Hence nets and filters are in many ways interchangable. Most of the propositions that I've proved using nets, can equally well be proven using filters. I've only considered nets, because they arize as a generalization of sequences, which was the starting point of this article. However it is instructive to also think about filters, as they give a diferent point of view. Only by combining both points of view, one can get the best insight into the mathematics.

Filters have another advantage. An easy application of Zorn's lemma gives that every filter on $X$ can be refined to an \textbf{ultrafilter}. An ultrafilter on $X$ is a filter which contains either $A$ or $X \setminus A$ for all $A \subseteq X$. The corresponding proposition for nets is that every net has a \textbf{universal} subnet. A net in $X$ is universal (also called an \textbf{ultranet}) if for every $A \subseteq X$, the net is either eventually in $A$ or eventually in $X \setminus A$. There is however no nice direct proof for the fact that every net has a universal subnet, filters are the more natural tool here.

From the fact that every filter has an ultrafilter refinement, an easy proof of Tychonoff's theorem is possible. Indeed, compact spaces can be characterized as those spaces where every ultrafilter has a limit. But then in a product of compact spaces, we can find a limit of any ultrafilter by considering the projections on every component and taking a limit in each of these compact spaces.

A proof of Tychonoff's theorem using just nets is also possible, but is not as elegant \cite{chernoff}.

Nets were introduced in 1922 by E.~H.~Moore and H.~L.~Smith in \cite{mooresmith}. Hence nets were at first called \emph{Moore-Smith sequences}. The theory of nets was further developed by Birkhoff \cite{birkhoff} (most of the propositions of this article first appear in his paper) and by Kelley \cite{kelley1950} (who introduces the terms \emph{net} and \emph{ultranet}, and who proves Tychonoff's theorem using ultrafilters). McShane \cite{mcshane} gives an extensive motivation for the definitions of nets and subnets.

Filters were introduced in 1937 by Cartan \cite{cartan1,cartan2}. Even though filters are now also used in very different contexts, Cartan's motivation for defining them was to generalize the notion of convergence for sequences. Indeed he starts of his article \emph{Th\'eorie des filtres} \cite{cartan1} by writing:
\begin{quote}
Malgr\'e les services rendus en topologie par la consideration des \emph{suites d\'enombrables}, leur emploi n'est pas adapt\'e \`a l'\'etude des espaces g\'en\'eraux. Nous voulons indiquer ici quel est l'instrument qui semble devoir les remplacer.\footnote{Translation from French: In spite of the accomplishments of considering countable sequences in topology, their use is not suitable in the study of general spaces. We want to indicate here which tool apparently should replace them.}
\end{quote}
Bourbaki's book on General Topology \cite{bourbaki} was the first to fully adopt the use of filters. The connections between nets and filters were investigated by Bartle \cite{bartle} and by Bruns and Schmidt \cite{bruns}.

\subsection{Sequential spaces and Fr\'echet-Urysohn spaces}

We have considered sequential spaces. Similar to sequential spaces are the Fr\'echet-Urysohn spaces. A topological space $X$ is \textbf{Fr\'echet-Urysohn} if the closure of any $A \subseteq X$ contains exactly the limits of sequences in $A$. They are also sometimes simply called Fr\'echet spaces, but this might cause confusion as there are other uses for the term \emph{Fr\'echet space}.

Any first countable space is Fr\'echet-Urysohn. This can be proven just like Proposition \ref{firstcountableseq}. The example from the proof of Proposition \ref{seqnotfirstcountable} shows that there is a Fr\'echet-Urysohn which is not first countable. Any Fr\'echet-Urysohn space is sequential, as by definition every sequentially closed set is its own closure. Fr\'echet-Urysohn are exactly those spaces of which every subspace is sequential. However not every subspace of a sequential space is sequential. Indeed, consider 
\[Y = \{(x,0) \st x \in \R \setminus \{0\} \}  \cup \{(0,1\} \cup \left\{ \left(\frac{1}{n+1},1\right) \st n \in \N \right\} \]
as a subset of the real plane. Let $(\R, \tau_q)$ be the quotient of $Y$ obtained by projecting onto the first coordinate. Then $(\R, \tau_q)$ is sequential as quotient of a metric space (Corollary \ref{seqquotientmetric}). But the subspace $\R \setminus \{ \frac{1}{n+1} \st n \in \N \}$ has a sequentially open set $\{0\}$ which is not open. So $(\R, \tau_q)$ is a sequential space that is not Fr\'echet-Urysohn.

Sequential spaces and Fr\'echet-Urysohn spaces where most intensively studied in the 1960s. Most of the results mentioned here where obtained by S.~P.~Franklin in \cite{franklin} and \cite{franklin2}. Sequential spaces and Fr\'echet-Urysohn spaces are also covered in Engelking's book \cite{engelking}.

\bibliographystyle{plain}
\bibliography{nets}

\end{document}